\documentclass[12pt, reqno]{amsart}
\usepackage{amsmath, amsthm, amscd, amsfonts, amssymb, graphicx, color}
\usepackage[bookmarksnumbered, colorlinks, plainpages]{hyperref}
\input{mathrsfs.sty}

\newtheorem{theorem}{Theorem}[section]
\newtheorem{lemma}[theorem]{Lemma}
\newtheorem{proposition}[theorem]{Proposition}
\newtheorem{corollary}[theorem]{Corollary}
\theoremstyle{definition}

\theoremstyle{remark}
\newtheorem{remark}[theorem]{Remark}
\numberwithin{equation}{section}

\begin{document}

\title{Operator Aczel inequality}

\author[M.S. Moslehian]{Mohammad Sal Moslehian}

\address{Department of Pure Mathematics, Center of Excellence in
Analysis on Algebraic Structures (CEAAS), Ferdowsi University of
Mashhad, P. O. Box 1159, Mashhad 91775, Iran.}
\email{moslehian@ferdowsi.um.ac.ir and moslehian@ams.org}
\urladdr{\url{http://profsite.um.ac.ir/~moslehian/}}



\begin{abstract}
We establish several operator versions of the classical Aczel inequality. One of operator versions deals with the weighted operator geometric mean and another is related to the positive sesquilinear forms. Some applications including the unital positive linear maps on $C^*$-algebras and the unitarily invariant norms on matrices are presented.
\end{abstract} \maketitle


\section{Introduction}
Let $\mathbb{B}(\mathscr{H})$ denote the algebra of all bounded linear operators acting on a complex Hilbert space $(\mathscr{H},\langle \cdot,\cdot\rangle)$ and $I$ is the identity operator. In the case where $\dim \mathscr{H} =n$, we identify $\mathbb{B}(\mathscr{H})$ with the full matrix algebra $M_n(\mathbb{C})$ of all $n \times n$ matrices with entries in the complex field $\mathbb{C}$. An operator $ A\in \mathbb{B}(\mathscr{H})$ is called positive (positive-semidefinite for matrices) if $\langle A\xi, \xi\rangle \geq 0$ holds for every $\xi\in \mathscr{H}$ and then we write $A\geq 0$. For $A,B \in\mathrm{B}(\mathcal{H})$, we say $A\leq B$ if $B-A\geq0$. Let $f$ be a continuous real valued function defined on an interval $J$. The function $f$ is called operator decreasing if $B\leq A$ implies $f(A)\leq f(B)$ for all $A, B $ with spectra in $J$. A function $f$ is said to be operator concave on $J$ if
$$\lambda f(A)+(1-\lambda)f(B) \leq f(\lambda A + (1-\lambda)B)$$
for all $A, B \in \mathbb{B}(\mathscr{H})$ with spectra in $J$ and all $\lambda \in [0,1]$.
\par By a $C^*$-algebra we mean a closed $*$-subalgebra $\mathscr{A}$ of $\mathbb{B}(\mathscr{H})$ for some Hilbert space $\mathscr{H}$. Any finite dimensional $C^*$-algebra is isometrically $*$-isomorphic to a direct sum of finitely many full matrix algebras. A $C^*$-algebra is called unital if it has an identity. A map taking the identity to identity is called unital. A map $\Phi: \mathscr{A} \to \mathscr{B}$ between $C^*$-algebras is called positive if it takes positive operators to positive ones, in particular, a positive linear map from $\mathscr{A}$ into $\mathbb{C}$ is called a positive linear functional. A map $\Phi: \mathscr{A} \to \mathscr{B}$ is called $2$-positive if the map $\Phi_2: M_2(\mathscr{A}) \to M_2(\mathscr{B})$ defined by $\Phi_2([a_{ij}])=[\Phi(a_{ij})]$ is positive, where $M_2(\mathscr{A})$ is the $C^*$-algebra all $2\times 2$ matrices with entries in $\mathscr{A}$. An operator $A$ is called a contraction if $\|A\|\leq 1$. We refer the reader to \cite{MUR} for undefined notions on operator theory and to \cite{F-M-P-S} for more information on
operator inequalities.

In 1956, Acz\'el \cite{ACZ} proved that if $a_i,b_i\,\,(1\leq i\leq n)$ are positive real numbers such that $a_1^2-\sum_{i=2}^na_i^2>0$ or $b_1^2-\sum_{i=2}^nb_i^2>0$, then $$\left(a_1b_1-\sum_{i=2}^na_ib_i\right)^2 \geq \left(a_1^2-\sum_{i=2}^na_i^2\right)\left(b_1^2-\sum_{i=2}^nb_i^2\right)\,.$$
The Acz\'el inequality has some applications in mathematical analysis and in the theory of functional equations in non-Euclidean geometry. During the last decades several interesting generalization of this significant inequality were obtained. Popoviciu \cite{POP} extended Acz\'el's inequality by showing that
$$\left(a_1b_1-\sum_{i=2}^na_ib_i\right)^p \geq \left(a_1^p-\sum_{i=2}^na_i^p\right)\left(b_1^p-\sum_{i=2}^nb_i^p\right)\,,$$ if $p\ge1$ and $a_1^p-\sum_{i=2}^na_i^p>0$ or $b_1^p-\sum_{i=2}^nb_i^p>0$. Acz\'el's inequality and Popoviciu's inequality was sharpened by Wu \cite{WU}, see also \cite{W-D}. A variant of Acz\'el's inequality in inner product spaces was given by Dragomir \cite{DRA1} by establishing that if $a$, $b$ are real numbers and $x$, $y$ are vectors of an inner product space such that $a^2-\|x\|^2>0$ or $b^2-\|y\|^2>0$, then $(a^2-\|x\|^2)(b^2-\|y\|^2)\leq (ab-\mbox{Re}\langle x,y\rangle)^2$, see also \cite{DRA2}. Cho, Mati\'c and Pe\v cari\'c \cite{C-M-P} generalized Acz\'el's inequality for linear isotonic functionals and convex functions. Several Acz\'el type inequalities involving norms in Banach spaces were presented by Mercer \cite{MER}. Also, Sun \cite{SUN} gave an Acz\'el--Chebyshev type inequality for positive linear functionals. To find operator versions of Hua's inequality (see \cite{MOS}),  Fujii \cite[Theorem 3]{FUJ} obtained an Acz\'el operator inequality by showing that if $\Phi: \mathscr{A} \to \mathscr{B}$ is a contractive $2$-positive unital linear map between unital $C^*$-algebras, $A, B \in \mathscr{A}$ are contraction and $\Phi(B^*A)$ is normal with the polar decomposition $\Phi(B^*A)=U|\Phi(B^*A)|$, then
$$|1-\Phi(B^*A)| \geq (1-\Phi(A^*A))\sharp U^*(1-\Phi(B^*B))U\,.$$
In this paper we establish several operator versions of the classical Aczel inequality. One of operator versions deals with the weighted operator geometric mean $A\sharp_{t}B:=A^{1/2}(A^{-1/2}BA^{-1/2})^tA^{1/2}\,\,(t \in [0,1])$ and another is related to the positive sesquilinear forms. Some applications including the unital positive linear maps on $C^*$-algebras and the unitarily invariant norms on matrices are presented. Recall that a unitarily invariant norm
$|||\cdot|||$ has the property $|||UXV|||=|||X|||$, where $U$ and $V$ are unitaries and $X \in M_n(\mathbb{C})$. For more information on the theory of
the unitarily invariant norms the reader is referred to \cite{BHA}.


\section{Operator Aczel inequality via geometric mean}

We start this section with a lemma about a parameterized operator power mean $m_t$ satisfying $Am_tB \leq (1-t)A+tB$ for any two positive invertible operators $A, B$. The power means $A\sharp_{r,t}B:=A^{1/2}(1-t+t(A^{-1/2}BA^{-1/2})^r)^{1/r}A^{1/2}\,\,(r\in[-1,1]\setminus\{0\})$ and  $A\sharp_{0,t}B:=A\sharp_{t}B\,\,(t \in [0,1])$ are such parameterized operator power means. Clearly if $AB=BA$, then $A\sharp_tB=A^{1-t}B^t$; see \cite[Chapter V]{F-M-P-S}.

\begin{lemma} \label{l1} Suppose that $m_t$ is a parameterized operator power mean not greater than the weighted arithmetic mean. If $J$ is an interval of $(0,\infty)$ and $f: J \to (0,\infty)$ is operator decreasing and operator concave on $J$ and $A, B \in \mathbb{B}(\mathscr{H})$ are positive invertible operators with spectra contained in $J$, then
\begin{eqnarray}\label{mt}
f(Am_tB)\geq f(A)m_tf(B)
\end{eqnarray}
\end{lemma}
\begin{proof}
It follows from $Am_tB \leq (1-t)A+tB$ that
\begin{eqnarray}\label{1}
f(Am_tB)&\geq& f(1-t)A+tB)\qquad\qquad\qquad\mbox{(since $f$ is operator decreasing)}\nonumber\\
&\geq&(1-t)f(A)+tf(B)\qquad\mbox{(since $f$ is operator concave)}\\
&\geq& f(A)m_tf(B)\qquad\qquad\qquad\qquad\qquad\qquad(\mbox{by the property of~} m_t)\,.\nonumber
\end{eqnarray}
\end{proof}

\begin{theorem} \label{t1} Let $J$ be an interval of $(0,\infty)$, let $f: J \to (0,\infty)$ be operator decreasing and operator concave on $J$, $\frac{1}{p}+\frac{1}{q}=1$, $p, q > 1$ and let $A, B \in \mathbb{B}(\mathscr{H})$ be positive invertible operators with spectra contained in $J$. Then
\begin{eqnarray}\label{main0}
f(A^p\sharp_{1/q}B^q)\geq f(A^p)\sharp_{1/q}f(B^q)
\end{eqnarray}
\begin{eqnarray}\label{main1}
\langle f(A^p\sharp_{1/q}B^q)\xi,\xi\rangle \geq \langle f(A^p)\xi,\xi\rangle^{\frac{1}{p}}\langle f(B^q)\xi,\xi\rangle^{\frac{1}{q}}\,.
\end{eqnarray}
for any vector $\xi \in \mathscr{H}$.
\end{theorem}
\begin{proof}
Lemma \ref{l1} yields inequality \eqref{main1}.

\noindent Let $\xi \in\mathscr{H}$ be an arbitrary vector. It follows from \eqref{1} that
\begin{eqnarray*}
\langle f(A^p\sharp_{1/q}B^q)\xi,\xi\rangle &\geq& \frac{1}{p}\langle f(A^p)\xi,\xi\rangle +\frac{1}{q}\langle f(B^q)\xi,\xi\rangle\\
&\geq& \langle f(A^p)\xi,\xi\rangle^{\frac{1}{p}}\langle f(B^q)\xi,\xi\rangle^{\frac{1}{q}}\\
&&\quad\quad\quad(\mbox{weighted arithmetic-geometric mean inequality}).
\end{eqnarray*}
\end{proof}
\begin{remark}
The H\"older--McCarthy inequality asserts that if $C \in \mathbb{B}(\mathscr{H})$ is a positive operator, then
$\langle C^r\xi,\xi\rangle \leq \langle C\xi,\xi\rangle^r$ for all $0 < r<1$ and all unit vectors $\xi\in\mathscr{H}$; cf.  \cite[Theorem 1.4]{F-M-P-S}. It follows from \eqref{main1} that
$$\langle f(A^p\sharp_{1/q}B^q)\xi,\xi\rangle \geq \langle f(A^p)^{\frac{1}{p}}\xi,\xi\rangle\langle f(B^q)^{\frac{1}{q}}\xi,\xi\rangle\,.$$
Thus
$$\|f(A^p\sharp_{1/q}B^q)^{1/2}\xi\| \geq \|f(A^p)^{\frac{1}{2p}}\xi\|\,\|f(B^q)^{\frac{1}{2q}}\xi\|\,.$$
\end{remark}

\begin{corollary} \label{c1} Let $\frac{1}{p}+\frac{1}{q}=1$, $p, q > 1$ and $A, B \in \mathbb{B}(\mathscr{H})$ be commuting positive invertible operators with spectra contained in $(0,1)$. Then
\begin{eqnarray*}
1-\|AB\xi\|^2 \geq \langle \left(1-\|A^{p/2}\xi\|^2\right)^{\frac{1}{p}}\left(1-\|B^{q/2}\|^2\right)^{\frac{1}{q}}\,.
\end{eqnarray*}
for any unit vector $\xi \in \mathscr{H}$.\\
\end{corollary}
\begin{proof}
Apply Theorem \ref{t1} to the function $f(t)=1-t$ on $(0,1)$ and note that $A^p\sharp_{1/q}B^q=AB$.
\end{proof}

\begin{corollary} \label{c4} Let $J$ be an interval of $(0,\infty)$, let $f: J \to (0,\infty)$ be operator decreasing and operator concave on $J$ and $A, B \in \mathbb{B}(\mathscr{H})$ be positive invertible commuting operators with spectra contained in $J$. Then
\begin{eqnarray*}
f(AB) \geq \left(f(A^p)\right)^{1/p}\left(f(B^q)\right)^{1/q}\,.
\end{eqnarray*}
\end{corollary}
\begin{corollary} \label{c2} If $f$ is a decreasing concave function on an interval $J$ and $a_i,b_i\,\,(1\leq i\leq n)$ are positive numbers in $J$, then
$$\sum_{i=1}^nf(a_ib_i) \geq \left(\sum_{i=1}^nf(a_i^p)\right)^{1/p}\left(\sum_{i=1}^nf(b_i^q)\right)^{1/q}\,.$$
\end{corollary}
Apply \eqref{main1} to the positive operators $A(x_1, \cdots,x_n)=(a_1x_1, \cdots, a_nx_n)$ and $B(x_1, \cdots,x_n)=(b_1x_1, \cdots, b_nx_n)$ acting on the Hilbert space $\mathscr{H}=\mathbb{C}^n$ and $\xi=(1, 1, \cdots, 1)$.


\section{Aczel's inequality via positive linear functionals}

In this section we present a new version of Aczel's inequality through positive linear functionals. The first result is a generalization of the main result of \cite{DRA1}. The proof differs from that the main result of \cite{DRA1}. It is an Acz\'el type inequality for sesquilinear forms.

\begin{theorem}\label{th2}
Let $\phi(.,.)$ be a positive sesquilinear form on a linear space $\mathscr{X}$, let $x, y \in \mathscr{X}$ such that $\phi(x,x) \leq M_1^2$ or $\phi(y,y) \leq M_2^2$ for some positive numbers $M_1, M_2$ and let $L: \mathbb{C} \to \mathbb{R}$ be a function fulfilling $L(z)\leq |z|$ for all $z\in \mathbb{C}$. Then
\begin{eqnarray*}
\big(M_1M_2-L(\phi(x,y))\big)^2 \geq \big(M_1^2-\phi(x,x)\big)\big(M_2^2-\phi(y,y)\big)\,.
\end{eqnarray*}
\end{theorem}
\begin{proof}
We may assume that both $\phi(x,x) \leq M_1^2$ and $\phi(y,y) \leq M_2^2$ hold. Then the Cauchy--Schwarz inequality implies that $|\phi(x,y)| \leq M_1M_2$. We have
\begin{align*}
&\left(M_1M_2-L(\phi(x,y))\right)^2 \geq \left(M_1M_2-|\phi(x,y)|\right)^2 \\
&\qquad\qquad\geq \left(M_1M_2- \sqrt{\phi(x,x)\phi(y,y)}\right)^2\qquad\qquad (\mbox{Cauchy--Schwarz inequality})\\
&\qquad\qquad\geq \left(M_1M_2- \frac{\phi(x,x)+\phi(y,y)}{2}\right)^2\quad (\mbox{arithmetic-geometric mean ineq.})\\
&\qquad\qquad\geq \left(\frac{M_1^2+M_2^2-\phi(x,x)-\phi(y,y)}{2}\right)^2\\
&\qquad\qquad\geq \left(M_1^2-\phi(x,x)\right)\left(M_2^2-\phi(y,y)\right)\quad(\mbox{arithmetic-geometric mean ineq.})
\end{align*}
as desired.
\end{proof}


The next result is a consequence of Theorem \ref{th2}, but we present a different proof for it.
\begin{theorem}
Suppose that $\Phi: \mathscr{A} \to \mathscr{B}$ is a unital positive linear map between unital $C^*$-algebras of operators acting on a Hilbert space $\mathscr{H}$, $A, B \in \mathscr{A}$ such that $\Phi(A^*A)$ or $\Phi(B^*B)$ is a contraction. Then
$$\big(1-\langle \Phi(B^*A)\xi,\xi\rangle\big)^2 \geq \big(1-\langle \Phi(A^*A)\xi,\xi\rangle\big)\big(1-\langle \Phi(B^*B)\xi,\xi\rangle\big)$$
for all unit vectors $\xi \in \mathscr{H}$.
\end{theorem}
\begin{proof} Without loss of generality, assume that $\Phi(A^*A)$ is a contraction. Let $\xi \in \mathscr{H}$ be a unit vector. Hence $\langle \Phi(A^*A)\xi,\xi\rangle \leq 1$. Let us consider the quadratic polynomial $$P(t)=(1-\langle \Phi(A^*A)\xi,\xi\rangle)t^2-2(1-\langle \Phi(B^*A)\xi,\xi\rangle)t+(1-\langle \Phi(B^*B)\xi,\xi\rangle)\,,$$
where $t\in\mathbb{R}$. It follows from the Cauchy--Schwarz inequality applied to the sesquilinear form $\langle A,B\rangle=\langle \Phi(B^*A)\xi,\xi\rangle$ that
$$P(1)=-\langle \Phi(A^*A)\xi,\xi\rangle + 2\langle \Phi(B^*A)\xi,\xi\rangle-\langle \Phi(B^*B)\xi,\xi\rangle\leq 0\,.$$
Clearly, $\lim_{t \to \infty} P(t)=\infty$. Hence the equation $P(t)=0$ has a root in $\mathbb{R}$. Thus
$$\big(1-\langle \Phi(B^*A)\xi,\xi\rangle\big)^2 - \big(1-\langle \Phi(A^*A)\xi,\xi\rangle\big)\big(1-\langle \Phi(B^*B)\xi,\xi\rangle\big)\geq 0$$
as desired.
\end{proof}
The next result is immediately deduced from Theorem \ref{th2}.
\begin{corollary}\label{cor1}
Let $\psi$ be a positive linear functional on a $C^*$-algebra $\mathscr{A}$, let $A, B \in \mathscr{A}$ such that $\psi(A^*A) \leq M_1^2$ or $\psi(B^*B) \leq M_2^2$ for some positive numbers $M_1, M_2$ and let $L: \mathbb{C} \to \mathbb{R}$ be a function fulfilling $L(z)\leq |z|$ for all $z\in \mathbb{C}$. Then
\begin{eqnarray}\label{psi}
\big(M_1M_2-L(\psi(B^*A))\big)^2 \geq \big(M_1^2-\psi(A^*A)\big)\big(M_2^2-\psi(B^*B)\big)\,.
\end{eqnarray}
\end{corollary}

\begin{corollary}[Acz\'el's Inequality] \label{c3} If $a_i, b_i\,\,(1\leq i\leq n)$ are positive numbers such that $\sum_{i=2}^na_i^2<1$ or $\sum_{i=2}^nb_i^2<1$, then
$$\left(1-\sum_{i=1}^na_ib_i\right)^2 \geq \left(1-\sum_{i=1}^na_i^2\right)\left(1-\sum_{i=1}^nb_i^2\right)\,.$$
\end{corollary}
Apply Corollary \ref{cor1} to the $n \times n$ matrices $A=\left[\begin{array}{cccc}a_1&&0\\ & \ddots &\\ 0&&a_n\end{array}\right]$ and $B=\left[\begin{array}{cccc}b_1&&0\\ & \ddots &\\ 0&&b_n\end{array}\right]$, positive linear functional ${\rm tr}(\cdot)$ on $M_n(\mathbb{C})$ and $L(z)=|z|$.\\

If we assume that $A$ and $B$ are normal contractions of a unital $C^*$-algebra $\mathscr{A}$, $AB=BA$ and consider the positive sesquilinear form $\phi(C,D)=\psi(D^*C)\,\,(C,D \in\mathscr{A})$, where $\psi$ is a pure state (or, equivalently, a non-zero complex homomorphism) on the commutative $C^*$-algebra generated by three elements $A, B$ and the identity $I$ of $\mathscr{A}$, then we get from \eqref{psi} that
\begin{eqnarray*}
\big(1-\mbox{Re}(\psi(B^*A))\big)^2 \geq \big(1-\psi(A^*A)\big)\big(1-\psi(B^*B)\big)\,,
\end{eqnarray*}
in which $\mbox{Re}$ denotes the real part. Hence
\begin{eqnarray*}
\psi\big(1-\mbox{Re}(B^*A)\big)^2 \geq \psi\big((1-A^*A)(1-B^*B)\big)\,,
\end{eqnarray*}
whence
\begin{eqnarray*}
\big(1-\mbox{Re}(B^*A)\big)^2 \geq (1-A^*A)(1-B^*B)\,.
\end{eqnarray*}
The same assertion is valid with the imaginary part $\mbox{Im}$ instead of $\mbox{Re}$. We proved therefore the following result.
\begin{corollary}
Let $A, B$ be commuting normal contractions of a unital $C^*$-algebra $\mathscr{A}$. Then
\begin{eqnarray*}
\big(1-\mbox{Re}(B^*A)\big)^2 \geq (1-A^*A)(1-B^*B)
\end{eqnarray*}
and
\begin{eqnarray*}
\big(1-\mbox{Im}(B^*A)\big)^2 \geq (1-A^*A)(1-B^*B)\,.
\end{eqnarray*}
\end{corollary}

\begin{corollary}
Let $\psi$ be a positive linear functional on $M_n(\mathbb{C})$, let $A, B \in M_n(\mathbb{C})$ such that $\psi(A)\leq M_1^2$ or $\psi(B) \leq M_2^2$ for some positive numbers $M_1, M_2$. Then
\begin{eqnarray*}
\big(M_1M_2-\psi(A\sharp B))\big)^2 \geq \big(M_1^2-\psi(A)\big)\big(M_2^2-\psi(B)\big)\,.
\end{eqnarray*}
\end{corollary}
\begin{proof}
We may assume that $\psi(A)\leq M_1^2$ and $\psi(B) \leq M_2^2$. The positive linear functional $\psi$ on $M_n(\mathbb{C})$ can be characterized by $\psi(C)=\langle CZ,Z\rangle$, where $Z\geq 0$ and $\langle \cdot,\cdot\rangle$ denotes the canonical inner product on $M_n(\mathbb{C})$ defined by $\langle X,Y\rangle={\rm tr}(Y^*X)$. It follows from \eqref{psi} with $L(z)=|z|$ and elements $A^{1/2}$ and
$(A^{-1/2}BA^{-1/2})^{1/2}A^{1/2}$ that
\begin{eqnarray*}
&\big(M_1M_2-\langle (A^{-1/2}BA^{-1/2})^{1/2}A^{1/2}Z, A^{1/2}Z\rangle\big)^2 \\
&\qquad \qquad\geq \big(M_1^2-\langle (A^{-1/2}BA^{-1/2})^{1/2}A^{1/2}Z, (A^{-1/2}BA^{-1/2})^{1/2}A^{1/2}Z\rangle\big)\\
&\qquad\qquad \times \big(M_2^2-\langle A^{1/2}Z, A^{1/2}Z\rangle\big)\,,
\end{eqnarray*}
or equivalently
\begin{eqnarray*}
&\big(M_1M_2-\langle A^{1/2}(A^{-1/2}BA^{-1/2})^{1/2}A^{1/2}Z, Z\rangle\big)^2 \\
&\qquad\qquad \geq \big(M_1^2-\langle A^{1/2}(A^{-1/2}BA^{-1/2})^{1/2}(A^{-1/2}BA^{-1/2})^{1/2}A^{1/2}Z, Z\rangle\big)\\
& \qquad\qquad\qquad\times \big(M_2^2-\langle A^{1/2}A^{1/2}Z, Z\rangle\big)\,,
\end{eqnarray*}
whence
\begin{eqnarray*}
\big(M_1M_2-\psi(A\sharp B))\big)^2 \geq \big(M_1^2-\psi(A)\big)\big(M_2^2-\psi(B)\big)\,.
\end{eqnarray*}
Note that $0\leq \psi(A\sharp B)=|\psi(A\sharp B)|$.
\end{proof}

Finally, we present an Acz\'el type inequality involving unitarily invariant norms. Note that
\begin{eqnarray}\label{uni1}
|||AXB|||\leq \|A\|\,|||X|||\,\|B\|
\end{eqnarray}
for all $X, A, B$. The arithmetic-geometric mean inequality states that
\begin{eqnarray}\label{uni2}
|||A^*XB||| \leq \frac{1}{2} |||AA^*X+XBB^*|||\,.
\end{eqnarray}
\begin{proposition}
Let $|||\cdot|||$ be a unitarily invariant norm on $M_n(\mathbb{C})$ and let $X, A, B \in M_n(\mathbb{C})$ such that $\|A\|^2\,|||X||| \leq 1$ or  $\|B\|^2\,|||X||| \leq 1$. Then
\begin{eqnarray*}
(1-|||A^*XB|||)^2 \geq (1-\|A\|^2\,|||X|||)(1-\|B\|^2\,|||X|||)
\end{eqnarray*}
\end{proposition}
\begin{proof}
\begin{align*}
&\quad\qquad\qquad(1-|||A^*XB|||)^2 \geq \left(1-\frac{1}{2}\left|||AA^*X+XBB^*\right|||\right)^2\quad\quad (\mbox{by~} \eqref{uni2})\\
&\quad\quad\geq \frac{(1-|||AA^*X|||)+(1-|||XBB^*|||)}{2}\qquad\qquad\quad\qquad\qquad(\mbox{triangle Ineq.})\\
&\quad\quad\geq (1-|||AA^*X|||)(1-|||XBB^*|||)\quad\quad(\mbox{arithmetic-geometric mean ineq.})\\
&\quad\quad\geq(1-\|A\|^2\,|||X|||)(1-\|B\|^2\,|||X|||)\qquad\qquad\qquad\qquad\qquad \quad(\mbox{by~} \eqref{uni1})\,.
\end{align*}
\end{proof}


\bibliographystyle{amsplain}

\end{document}